\documentclass[11pt, reqno]{amsart}
\usepackage{amsmath,amssymb,amstext}
\usepackage{amsthm}
\usepackage{stmaryrd} % For the \mapsfrom arrow
\usepackage{mathdots} % For the \iddots dots
\usepackage{mathtools} % For the \mathclap command
\usepackage{verbatim}
\usepackage{enumerate}
\usepackage{graphicx}
\graphicspath{{images/}}
\usepackage{wrapfig}
\usepackage{caption}
\usepackage[utf8]{inputenc}
\usepackage[T1]{fontenc}
\usepackage[english]{babel}
\usepackage{pgfplots}
\usepackage{tikz-cd}
\usepackage{xcolor}
\usepackage{mdframed}
\usepackage{empheq}
\usepackage{float}
\usepackage[textsize=footnotesize, color=white, bordercolor=white]{todonotes}
\usepackage{rotating} % To rotate a table
\usepackage{etoolbox} % To reduce ToC space between sections
\usepackage{thmtools} % To restate theorems
\usepackage{thm-restate}
\usepackage{pdfpages}
\usepackage{pinlabel}
\usepackage{calc}
\usepackage[pagebackref]{hyperref}
\usepackage[nameinlink]{cleveref}

\let\cref\Cref
\crefname{subsection}{subsection}{subsections}
\Crefname{subsection}{Subsection}{Subsections}
\crefname{thm}{theorem}{theorem}
\crefname{dfn}{definition}{definition}
\crefname{prop}{proposition}{proposition}
\crefname{lem}{lemma}{lemma}
\crefname{cor}{corollary}{corollary}
\crefname{ex}{example}{example}
\Crefname{enumi}{}{}
\Crefname{equation}{}{}
\definecolor{darkblue}{RGB}{0,0,96}
\definecolor{gray}{RGB}{127,127,127}
\definecolor{darkred}{RGB}{160,0,0}
\definecolor{lightyellow}{RGB}{255,255,128}
\hypersetup{colorlinks={true},linkcolor={darkblue},citecolor={darkblue},filecolor={darkblue},urlcolor={darkblue},pdfauthor={Damian Iltgen},pdftitle={Lower bound on the stable $4$ genus of knots}}

\newcommand{\qua}{\hskip 0.4em \ignorespaces}
\def\arxiv#1{\relax\ifhmode\unskip\qua\fi
\href{http://arxiv.org/abs/#1}%
{\tt arXiv:\penalty -100\unskip#1}}

\def\ZB#1{\relax\ifhmode\unskip\qua\fi
\href{https://zbmath.org/?q=an:#1}{\tt zb#1}}
\def\xox#1{\csname xx#1\endcsname}

\renewcommand*{\backrefalt}[4]{%
\textcolor[rgb]{0,0,0}{\small
\ifcase #1 %
No citations.%
\or
Cited on page~#2.%
\else
Cited on pages~#2.%
\fi
}}

\newtheoremstyle{mytheorem}{}{}{\upshape}{}{\bfseries}{.}{ }{\thmname{#1}\thmnumber{ #2}\thmnote{ (#3)}}

\theoremstyle{mytheorem}
\newtheorem{thm}{Theorem}
\newtheorem{dfn}[thm]{Definition}

\newtheorem*{rem}{Remark}
\newtheorem{prop}[thm]{Proposition}

\newtheorem{cor}[thm]{Corollary}

\newtheorem*{prf}{Proof}
\newtheorem*{prob}{Problem}

\newtheorem{manualtheoreminner}{\hspace{-0.15cm}}
\newenvironment{manualtheorem}[1]{%
  \renewcommand\themanualtheoreminner{#1}%
  \manualtheoreminner
}{\endmanualtheoreminner} % For a theorem environment with manual theorem number

\newtheorem{manualpropositioninner}{\hspace{-0.15cm}}
\newenvironment{manualproposition}[1]{%
  \renewcommand\themanualpropositioninner{#1}%
  \manualpropositioninner
}{\endmanualpropositioninner} % For a proposition environment with manual proposition number

\newtheorem{manualcorollaryinner}{\hspace{-0.15cm}}
\newenvironment{manualcorollary}[1]{%
  \renewcommand\themanualcorollaryinner{#1}%
  \manualcorollaryinner
}{\endmanualcorollaryinner} % For a corollary environment with manual corollary number

\newcommand{\vertiii}[1]{{\left\vert\kern-0.25ex\left\vert\kern-0.25ex\left\vert #1 \right\vert\kern-0.25ex\right\vert\kern-0.25ex\right\vert}} % For the triple-bar norm
\makeatletter
\newcommand{\subalign}[1]{%
  \vcenter{%
    \Let@ \restore@math@cr \default@tag
    \baselineskip\fontdimen10 \scriptfont\tw@
    \advance\baselineskip\fontdimen12 \scriptfont\tw@
    \lineskip\thr@@\fontdimen8 \scriptfont\thr@@
    \lineskiplimit\lineskip
    \ialign{\hfil$\m@th\scriptstyle##$&$\m@th\scriptstyle{}##$\hfil\crcr
      #1\crcr
    }%
  }%
}
\makeatother % For alignment in substack
\newcommand{\myqed}{\pushQED{\qed}\qedhere} % For the QED symbol appearing on the same line as an equation

\title{A Lower Bound on the Stable $4$-Genus of Knots}
\author{Damian Iltgen}
\date{}
\address{Fakultät für Mathematik, Universität Regensburg, Germany}
\email{damian.iltgen@ur.de}

\begin{document}

\maketitle

\begin{abstract}
	We present a lower bound on the stable $4$-genus of a knot based on Casson-Gordon $\tau$-signatures. We compute the lower bound for an infinite family of knots, the twist knots, and show that a twist knot is torsion in the knot concordance group if and only if it has vanishing stable $4$-genus.
\end{abstract}

\section{Introduction}

In 2010, Charles Livingston \cite{livstable} introduced a new knot invariant called the \textit{stable $4$-genus $g_{st}$} of a knot $K$, which is defined as
\begin{equation*}
	g_{st}(K) = \lim_{n\rightarrow \infty} \frac{g_4(nK)}{n},
\end{equation*}
where $nK$ denotes the $n$-fold connected sum $K\# \cdots \# K$ and $g_4$ is the (topological) $4$-genus.\footnote{Throughout this paper, we will work in the topological category.} Recall that the $4$-genus of a knot $K$ is defined as the minimal genus over all properly embedded and locally flat surfaces $\Sigma \subset B^4$ with $\partial\Sigma = K$. In general, it is rather difficult to compute the stable $4$-genus of a knot. Most of the knot invariants that give bounds on the $4$-genus, such as the Levine-Tristram signatures \cite{levine1,tristram}, are additive under connected sum, hence they cannot yield bounds that are different for $g_{st}$.

More promising are Casson-Gordon invariants \cite{cg1, cg2}. For instance, Livingston used them in \cite{livstable} to show that a specific satellite construction yields knots whose stable $4$-genus is close to but not greater than~$1/2$. Note that it is an open question whether there exists a knot $K$ such that $0 < g_{st}(K) < 1/2$. In this paper, we use Casson-Gordon invariants once more to construct a lower bound on $g_{st}$. Our results show that already a simple family of knots, the twist knots, contains an infinite subfamily with stable $4$-genus close to but not greater than $1/2$. The main results are as follows.

\begin{manualtheorem}{\cref{mainthm}}[Main Theorem]
	Let $K$ be a knot with $d$-fold branched cover $X_d$ where $d$ is a prime-power, and let $p$ be any prime. If the rational numbers $L_1, L_2, \dots, L_m$ (defined below) have the same sign, and if $\sum_{s=1}^{d-1} \sigma_{s/d}(K) = 0$, where $\sigma_{s/d}(K)$ is the Levine-Tristram signature of $K$ associated to $e^{2\pi is/d}$, then
	\begin{equation*}
		g_{st}(K) \geq \frac{t\cdot L}{4d(p-1)+2(d-1)L},
	\end{equation*}
	where $t \coloneqq \dim H^1(X_d; \mathbb{F}_p)$ and $L \coloneqq \min_{j=1,\dots,m} \vert L_j\vert$.
\end{manualtheorem}

Here, the numbers $L_j$ for $j = 1, \dots, m$, are defined as
\begin{equation*}
	L_j \coloneqq \sum_{\chi\in A_j} \sigma_1(\tau(K, \chi)) \in \mathbb{Q},
\end{equation*}
where $A_1, \dots, A_m \subset H^1(X_d; \mathbb{F}_p)$ are the one-dimensional subspaces of\break$H^1(X_d; \mathbb{F}_p)$, and $\sigma_1(\tau(K, \chi))$ is the Casson-Gordon $\tau$-signature corresponding to $\chi$ (see \cref{seccginv} and \cref{maindef}).

The strength of the bound in \cref{mainthm} depends on the choice of $d$ and $p$. However, $t\cdot L/(4d(p-1)+2(d-1)L)$ is bounded from above by $t/(2(d-1))$, so a priori the best bounds are obtained in the case $d = 2$, i.e.\ when working with the double branched cover $X_2$. Note that if one of the numbers $L_j$ is zero, then also $L = 0$ and \cref{mainthm} will yield the trivial bound $g_{st}(K) \geq 0$.

Note that in \cref{mainthm}, we make the assumption that the sum of Levine-Tristram signatures $\sum_{s=1}^{d-1} \sigma_{s/d}(K)$ of $K$ vanishes. This is needed so that one of our main tools in the proof -- Gilmer's lower bound on the $4$-genus given by Casson-Gordon $\tau$-signatures (see \cref{Gilmer} and \cref{Gilmer2} in \cref{seccginv}) -- admits a more simplified application. The given proof of \cref{mainthm} doesn't hold without this assumption, and we currently do not know if of our methods generalize to the case where $\sum_{s=1}^{d-1} \sigma_{s/d}(K)$ doesn't vanish. However, if any of the Levine-Tristram signatures $\sigma_{s/d}(K)$ is non-zero, then there is the Murasugi-Tristram bound $g_4(nK) \geq \frac{n}{2}\vert\sigma_{s/d}(K)\vert$ \cite{murasugi3, tristram} which implies $g_{st}(K) \geq \frac{1}{2}\vert\sigma_{s/d}(K)\vert$, so we still obtain a non-trivial lower bound on $g_{st}$ when $\sum_{s=1}^{d-1} \sigma_{s/d}(K)$ doesn't vanish.\footnote{In earlier versions of the article on the arXiv, and in the published version in \emph{Algebraic \& Geometric Topology 22-5 (2022)}, \cref{mainprop} and \cref{mainthm} are falsely stated without the assumption that the sum of Levine-Tristram signatures vanishes.}

The main theorem is an immediate consequence of the following proposition.

\begin{manualproposition}{\cref{mainprop}}
	With the same assumptions as in \cref{mainthm},
	\begin{equation*}
		g_4(nK) \geq \frac{nt\cdot L}{4d(p-1)+2(d-1)L}
	\end{equation*}
	for any $n \in \mathbb{N}$.
\end{manualproposition}

For the reader who would like to skip the (quite lengthy) proof of the main theorem without missing out on the key proof technique, we will provide a quick proof of the following introductory result.

\begin{manualproposition}{\cref{propconcorder}}
	Let $K$ be a knot and $p$ a prime such that $H^1(X_d; \mathbb{F}_p)$ is one-dimensional. If $L = \vert L_1\vert > 0$, then $g_4(nK) \neq 0$ for all $n \in \mathbb{N}$.
\end{manualproposition}

%Observe that in general, whenever the stable $4$-genus of a knot $K$ is non-zero, then $K$ has infinite order in the topological concordance group $\mathcal{C}$.

As a sample application, we compute the lower bound given by \cref{mainthm} with $d = 2$ for an infinite family of knots, the twist knots $K_n$ (see \cref{fig:twist_knots_with_brace_intro}). The result is as follows.

\begin{figure}[H]
	\centering
	\captionsetup{font=small}
	\def\svgscale{0.24}
	%% Creator: Inkscape 1.0beta2 (2b71d25, 2019-12-03), www.inkscape.org
%% Accompanies image file 'twist_knots_with_brace.pdf' (pdf, eps, ps)
\begingroup%
  \makeatletter%
  \providecommand\color[2][]{%
    \errmessage{(Inkscape) Color is used for the text in Inkscape, but the package 'color.sty' is not loaded}%
    \renewcommand\color[2][]{}%
  }%
  \providecommand\transparent[1]{%
    \errmessage{(Inkscape) Transparency is used (non-zero) for the text in Inkscape, but the package 'transparent.sty' is not loaded}%
    \renewcommand\transparent[1]{}%
  }%
  \providecommand\rotatebox[2]{#2}%
  \newcommand*\fsize{\dimexpr\f@size pt\relax}%
  \newcommand*\lineheight[1]{\fontsize{\fsize}{#1\fsize}\selectfont}%
  \ifx\svgwidth\undefined%
    \setlength{\unitlength}{736.59649658bp}%
    \ifx\svgscale\undefined%
      \relax%
    \else%
      \setlength{\unitlength}{\unitlength * \real{\svgscale}}%
    \fi%
  \else%
    \setlength{\unitlength}{\svgwidth}%
  \fi%
  \global\let\svgwidth\undefined%
  \global\let\svgscale\undefined%
  \makeatother%
  \begin{picture}(1,0.57307188)%
    \lineheight{1}%
    \setlength\tabcolsep{0pt}%
    \put(0,0){\includegraphics[width=\unitlength,page=1]{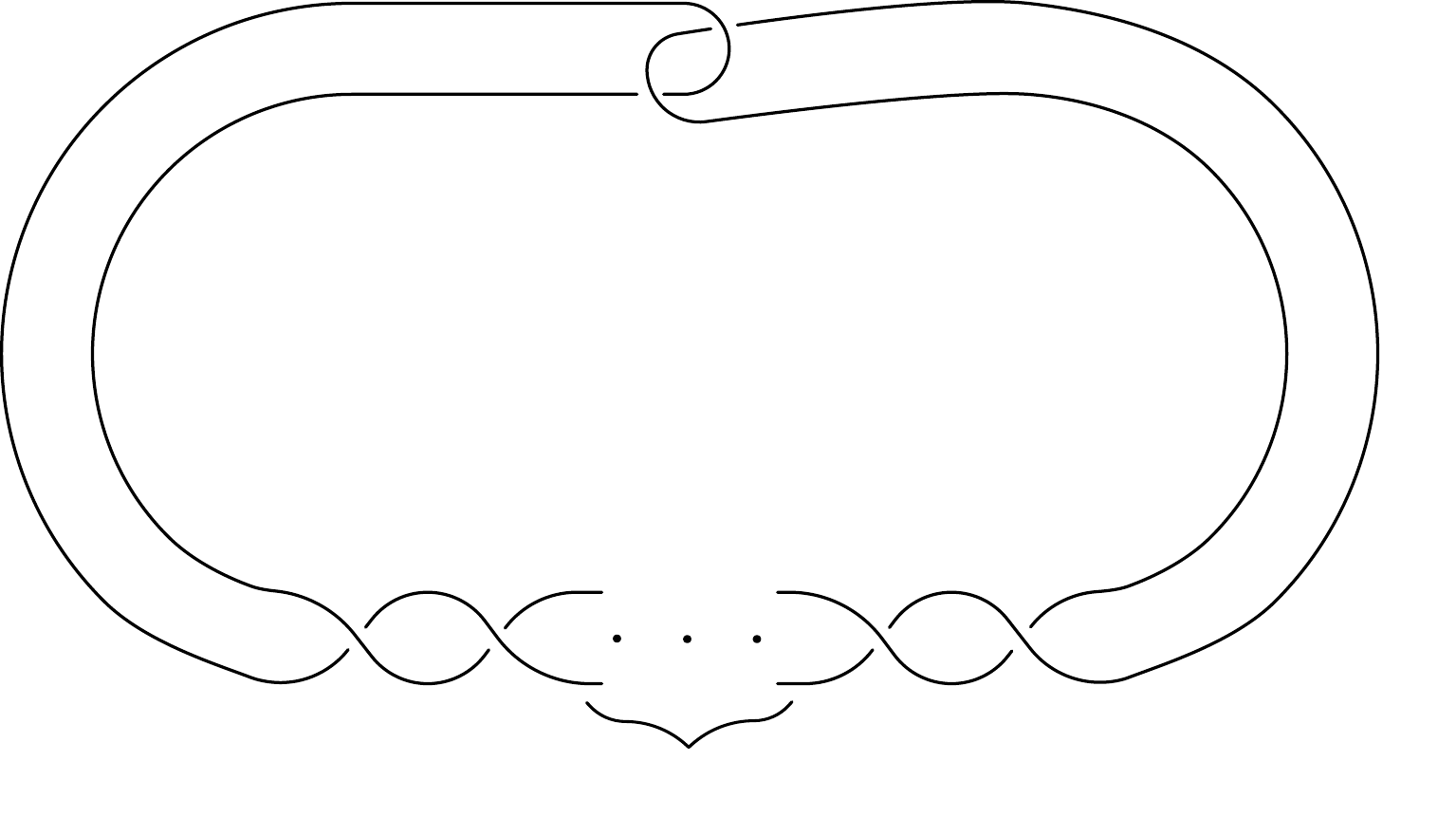}}%
    \put(0.31360419,0.00288182){\color[rgb]{0,0,0}\makebox(0,0)[lt]{\lineheight{1.25}\smash{\begin{tabular}[t]{l}$n$ full twists\end{tabular}}}}%
    \put(0.9102677,0.51605288){\color[rgb]{0,0,0}\makebox(0,0)[lt]{\lineheight{1.25}\smash{\begin{tabular}[t]{l}$K_n$\end{tabular}}}}%
  \end{picture}%
\endgroup%

	\caption{The twist knot $K_n$}
	\label{fig:twist_knots_with_brace_intro}
\end{figure}

\begin{manualcorollary}{\cref{corexmain}}
	Let $K_n$ be the twist knot with $n \in \mathbb{N}\backslash\{0,2\}$ full right hand twists and $p$ a prime dividing $4n+1$. Then
	\begin{equation*}
		g_{st}(K_n) \geq \begin{cases}
		\frac{(pq+q-6)}{2(pq+q+18)}, & (p-1)/2 \textrm{ even} \\
		\frac{p^2q-6p-q-6}{2(p^2q+18p-q-30)}, & (p-1)/2 \textrm{ odd},
	\end{cases}
	\end{equation*}
	where $q = (4n+1)/p$.
\end{manualcorollary}

While \cref{corexmain} is directly obtained from the main theorem, the bounds are not particularly easy to grasp. By estimating further from below, we obtain a single bound which is weaker but easier to grasp and holds for all twist knots $K_n$ simultaneously.

\begin{manualcorollary}{\cref{corexadd}}[\cref{corexmain}, weakened]
	Let $K_n$ with $n \in \mathbb{N}$ be any twist knot. Then
	\begin{equation*}
		g_{st}(K_n) \geq \frac{1}{2} - \frac{6}{2n+7}.
	\end{equation*}
\end{manualcorollary}

It is straightforward to see that for growing $n$, the bound given in \cref{corexadd}, and in fact also the stronger bounds in \cref{corexmain}, tend towards $1/2$. This means in particular that our bounds for the twist knots lie in the interval $[0, \frac{1}{2})$. Since the Levine-Tristram signatures of all twist knots vanish, these are the best bounds that we currently know.

Note that the first three twist knots $K_0$, $K_1$ and $K_2$ form a special case. Casson and Gordon proved \cite{cg1, cg2} that the unknot $K_0$ and the so-called Stevedore knot $K_2$ are the only slice knots among the twist knots, and so the stable $4$-genus of $K_0$ and $K_2$ vanishes. The knot $K_1$ represents the figure-eight knot, and it is well known that the figure-eight has order $2$ in the knot concordance group $\mathcal{C}$. It follows that the stable $4$-genus of $K_1$ vanishes as well. This coincides with the trivial bound obtained from \cref{corexmain} in those cases. In fact, $K_0$, $K_1$ and $K_2$ are the only twist knots with $L = 0$ (see \cref{secex}), which means that for any other twist knot, a non-trivial lower bound can be obtained from \cref{corexmain} (resp.\ \cref{mainthm}). This establishes the following corollaries.

\begin{manualcorollary}{\cref{corextorsion}}
	Let $K_n$ be any twist knot. Then
	\begin{equation*}
		K_n \textrm{ is torsion in $\mathcal{C}$} \iff g_{st}(K_n) = 0.
	\end{equation*}
\end{manualcorollary}

\begin{manualcorollary}{\cref{corexconcorder}}
	$K_n$ is of infinite order in $\mathcal{C}$ except for $n = 0, 1, 2$.
\end{manualcorollary}

It is not known whether \cref{corextorsion} holds for arbitrary knots $K$.

\begin{prob}[Livingston \cite{livstable}]
	Given a knot $K$, does $g_{st}(K) = 0$ imply that $K$ is torsion in the knot concordance group $\mathcal{C}$?
\end{prob}

\cref{corexconcorder} gives a full description of the concordance order of twist knots in the topological category. This result was to the best knowledge of the author not completely known before, despite the vast work that has been done on the topological concordance order of twist knots in the past. We would like to mention at this point some of the previously obtained results. First and foremost, Casson and Gordon \cite{cg1, cg2} showed that if $K_n$ is algebraically slice and $n \neq 0, 2$, then $K_n$ has infinite order in the knot concordance group $\mathcal{C}$. Livingston and Naik \cite{lntor1, lntor2} obtained the same result for all twist knots of algebraic concordance order $4$. The remaining case of algebraic order $2$ was partially solved by Davis \cite{davistor} and Tamulis \cite{tamulis}, who obtained infinite concordance order of $K_n$ if $n = x^2+x+1$ with $x \geq 2$ \cite[Corollary 6]{davistor}, or if $4n+1$ is prime with $n \geq 3$ \cite[Theorem 1.1]{tamulis}. With our \cref{corexconcorder}, all other cases are now solved.

A recent result by Baader and Lewark \cite{bls4g} shows that $g_{st}(K_n) \leq 2/3$ for every $n \in \mathbb{N}$. In fact, this bound can be improved for certain $n$ as will be shown in \cref{secexupperbound}.

\begin{manualproposition}{\cref{propex}}
	Let $n \in \mathbb{N}$ such that the negative Pell equation\break $x^2 - (4n+1)y^2 = -1$ has a solution $x, y \in \mathbb{Z}$. Then
	\begin{equation*}
		g_{st}(K_n) \leq \frac{1}{2}.
	\end{equation*}
\end{manualproposition}

As shown by Rippon and Taylor in \cite{rtpell}, the negative Pell equation $x^2 - (4n+1)y^2 = -1$ has a solution if and only if the continued fraction of $\sqrt{4n+1}$ has odd period length. This is the case, for example, if $4n+1 = p^k$, where $p$ is a prime such that $p \equiv 1 \mod 4$ and $k \in \mathbb{N}$ is odd \cite{rtpell}. This yields the infinite family of twist knots whose stable $4$-genus is close to but not greater than $1/2$.

\subsection{Smooth vs.\ topological setting}\label{secsmoothtop}

We would like to make a short remark about the situation in the smooth setting. Since the smooth stable $4$-genus is always greater than or equal to the topological stable $4$-genus, our result also applies in the smooth setting. When it comes to potential lower bounds obtained by smooth knot invariants such as the Rasmussen $s$- and $\tau$-invariant, or the Ozsváth-Stipsicz-Szabó $\Upsilon$-invariant, we face the same problem as with the Levine-Tristram signatures in the topological setting: all these invariants are additive under connected sum. In particular, no better lower bounds in the smooth setting are known.

Regarding the results about twist knots, \cref{corexmain} still holds in the smooth setting. This is no longer true for \cref{propex}: the upper bound is obtained by using machinery that is exclusive to the topological setting and is therefore no longer valid. However, a result by Baader and Lewark \cite{bls4g} implies that the smooth stable $4$-genus of twist knots is strictly smaller than $1$, so we can still say that it is contained in the interval $[0, 1)$. Since our lower bounds also hold in the smooth setting, we can further say that $K_0$, $K_1$ and $K_2$ are the only twist knots with smooth stable $4$-genus equal to zero. Other information is not known. In particular, we cannot tell whether the topological and smooth stable $4$-genus of twist knots coincide or not, except in the case of $K_0$, $K_1$ and $K_2$ where they are the same. This raises the following open question.

\begin{prob}
	Compute the smooth and topological stable $4$-genus for any twist knot $K_n$ with $n \geq 3$.
\end{prob}

Twist knots form one of the simplest family of knots. The fact that we don't know the exact value of the stable $4$-genus for any twist knot (other than the first three) shows how difficult it is to determine $g_{st}$ in general. Moreover, to the best knowledge of the author all known exact values of the stable $4$-genus are integers. This raises further the following question.

\begin{prob}
	Does there exist either in the smooth or topological setting a knot $K$ such that $g_{st}(K) \notin \mathbb{Z}$ and for which the exact value of the stable $4$-genus can be computed?
\end{prob}

We would like to conclude this discussion with a remark regarding \cref{corexconcorder}. In the smooth category, Lisca \cite{lisca} gave a complete description of the concordance order of $2$-bridge knots, a class to which the twist knots belong. It follows from his result that the only twist knots of finite order in the smooth concordance group are $K_0$, $K_1$, and $K_2$. This coincides with the results obtained in the present work and shows that the concordance order of the twist knots is the same in the topological and smooth category.

\subsection{Organization}\label{secorganization}

The paper is structured as follows. In \cref{secprelim}, we state the tools and definitions needed in later parts. \cref{secmain} forms the main part of the paper and is occupied with the proof of \cref{mainprop} and \cref{mainthm}. Here we also prove \cref{propconcorder}. In \cref{secex}, we compute the lower bound for $g_{st}$ for the twist knots $K_n$ and show that there is an infinite subfamily with stable $4$-genus close to but not greater than $1/2$.

\section*{Acknowledgments}

The author would like to sincerely thank his advisor Lukas Lewark for helpful discussions and remarkable advice during the process of writing this paper. The author would further like to thank Anthony Conway, Stefan Friedl, Charles Livingston and Allison Miller for reading a preliminary version of the paper and giving final comments and advice. The author is supported by the Emmy Noether Programme of the DFG, project number 412851057.

\section{Preliminaries}\label{secprelim}

The purpose of this section is to briefly review the definitions and notions of the objects that are used throughout the paper. References include \cite{cg1, cg2, conwaycg, gilmer3, gilmer2, livnaik}.

\subsection{Linking forms and metabolizers}\label{seclfm}

We follow the definitions of \cite{gilmer2}. Let $G$ be a finite abelian group. A \textit{linking form} on $G$ is a symmetric bilinear map
\begin{equation*}
	\alpha\colon G \times G \rightarrow \mathbb{Q}/\mathbb{Z}
\end{equation*}
which is non-degenerate, i.e.\ the correlation map $c\colon G \rightarrow \textrm{Hom}(G, \mathbb{Q}/\mathbb{Z})$ is an isomorphism. If $H \subseteq G$ is a subgroup, define
\begin{equation*}
	H^\perp \coloneqq \{g \in G \ | \ \alpha(g, h) = 0 \textrm{ for all } h \in H\}.
\end{equation*}
If there is a subgroup $H \subseteq G$ such that $H = H^\perp$, then $\alpha$ is called \textit{metabolic} and $H$ is called a \textit{metabolizer}. If $H$ is a metabolizer for $\alpha$, then
\begin{equation*}
	\vert H\vert^2 = \vert G\vert,
\end{equation*}
which is a consequence of $H = H^\perp$ and $c$ being an isomorphism.

It will be convenient for our purposes to consider the above in a slightly different setting. Given a prime $p$, let $V$ be a finite-dimensional vector space over $\mathbb{F}_p$, the finite field with $p$ elements. If $V$ can be embedded in $G$ as a finite abelian group, then $\alpha$ induces a linking form $\overline\alpha$ on $V$ in the following way. After choosing an embedding $\varphi\colon V \hookrightarrow G$, we define $\overline\alpha$ to be the restriction of $\alpha$ to the subgroup $\varphi(V) \subseteq G$. In other words,
\begin{equation*}
	\overline\alpha \coloneqq \alpha \circ (\varphi\times \varphi).
\end{equation*}
Note that the subgroup $\varphi(V)$ consists only of $p$-torsion. In general, the induced form $\overline\alpha$ depends on the chosen embedding $\varphi$ and will no longer be non-degenerate. However, any two forms obtained in this way are isometric provided that their domains are the same: if $\varphi_1, \varphi_2\colon V \hookrightarrow G$ are two embeddings with $\varphi_1(V) = \varphi_2(V)$ inducing the forms $\overline\alpha_1$ and $\overline\alpha_2$, then
\begin{equation*}
	f\colon (V, \overline\alpha_1) \rightarrow (V, \overline\alpha_2),\quad f = \varphi_2^{-1} \circ \varphi_1
\end{equation*}
defines an isometry between them. Given such an induced form $\overline\alpha$, we can similarly define $F^\perp$ for a given subspace $F \subseteq V$. If $F$ is such that $F = F^\perp$, then $F$ will be called a \textit{generalized metabolizer} for $\overline \alpha$. If $F$ is a generalized metabolizer, then
\begin{equation*}
	2\dim F \geq \dim V.
\end{equation*}
In particular, a generalized metabolizer $F$ can consist of the entire space $V$.

\subsection{Casson-Gordon invariants and $\tau$-signatures}\label{seccginv}

Let $K$ be a knot with $d$-fold branched cover $X_d$, where $d$ is a power of a prime. Given a character $\chi\colon H_1(X_d; \mathbb{Z}) \rightarrow \mathbb{Q}/\mathbb{Z}$ of prime-power order, Casson and Gordon \cite{cg1, cg2} defined in the 70's an invariant $\tau(K, \chi)$ which is an element in the rationalized Witt group $W(\mathbb{C}(t)) \otimes \mathbb{Q}$ (see for instance \cite{cg1, cg2, conwaycg} for the definition). This invariant yields an obstruction to the sliceness of a knot \cite{cg1}, and one of the most famous applications is the proof that the only slice knots among the twist knots are the unknot and the Stevedore knot \cite{cg1, cg2}.

Given $\lambda \in S^1$, there is a signature homomorphism $\sigma_\lambda\colon W(\mathbb{C}(t)) \otimes \mathbb{Q} \rightarrow \mathbb{Q}$ (see \cite{cg1, conway}) that can be used to obtain further results about the $4$-genus of a knot when applied to the Casson-Gordon invariant $\tau(K, \chi)$.

In order to state one of our main tools, we need to define a certain linking form $\beta$ on $H^1(X_d; \mathbb{Q}/\mathbb{Z})$ which is closely related to the geometric linking form $l$ on $H_1(X_d; \mathbb{Z})$. The definition is as follows. Let $\xi\colon H^1(X_d; \mathbb{Z}) \rightarrow H^2(X_d; \mathbb{Q}/\mathbb{Z})$ be the Bockstein homomorphism associated to the short exact sequence $0 \rightarrow \mathbb{Z} \rightarrow \mathbb{Q} \rightarrow \mathbb{Q}/\mathbb{Z} \rightarrow 0$ of coefficients in cohomology. Note that $\xi$ is in fact an isomorphism since $X_d$ is a rational homology sphere. Let $\Phi$ be the composition
\begin{equation*}
	H_1(X_d; \mathbb{Z}) \overset{\textrm{PD}^{-1}}\rightarrow H^2(X_d; \mathbb{Z}) \overset{\xi^{-1}}\rightarrow H^1(X_d; \mathbb{Q}/\mathbb{Z}) \overset{\textrm{ev}}\rightarrow \textrm{Hom}(H_1(X_d; \mathbb{Z}), \mathbb{Q}/\mathbb{Z}),
\end{equation*}
where $\textrm{PD}$ is Poincaré duality, $\xi^{-1}$ is the inverse of the Bockstein map, and $\textrm{ev}$ is the Kronecker evaluation map (see \cite{cfhlf} for more details). The \textit{geometric linking form} $l$ is defined as
\begin{equation*}
	l\colon H_1(X_d; \mathbb{Z}) \times H_1(X_d; \mathbb{Z}) \rightarrow \mathbb{Q}/\mathbb{Z},\quad l(a,b) = \Phi(a)(b).
\end{equation*}
Finally, we define $\beta$ to be the form
\begin{equation*}
	\beta\colon H^1(X_d; \mathbb{Q}/\mathbb{Z}) \times H^1(X_d; \mathbb{Q}/\mathbb{Z}) \rightarrow \mathbb{Q}/\mathbb{Z},\quad \beta(x,y) = -l(\nu(x), \nu(y)),
\end{equation*}
where $\nu \coloneqq \textrm{PD}\circ\xi$. Since $H^1(X_d; \mathbb{Q}/\mathbb{Z}) \cong \textrm{Hom}(H_1(X_d; \mathbb{Z}), \mathbb{Q}/\mathbb{Z})$ by universal coefficients, $\beta$ can be seen as the dual form of the geometric linking form $l$ on $(H^1(X_d; \mathbb{Z}))^* = \textrm{Hom}(H_1(X_d; \mathbb{Z}), \mathbb{Q}/\mathbb{Z})$. The following theorem is due to Gilmer \cite{gilmer2}.

\begin{thm}[Gilmer]\label{Gilmer}
	Let $K$ be a knot with $g_4(K) = g$. Then\break $(H^1(X_d; \mathbb{Q}/\mathbb{Z}), \beta)$ splits as direct sum $(B_1 \oplus B_2, \beta_1 \oplus \beta_2)$ such that
	\begin{enumerate}
		\item $\beta_1$ has an even presentation with rank $2(d-1)g$ and signature \break $\sum_{s=1}^{d-1} \sigma_{s/d}(K)$;
		\item $\beta_2$ has a metabolizer $H$ such that for every $\chi \in H$ of prime-power order,
			\begin{equation*}
				\vert \sigma_1(\tau(K, \chi)) + \sum_{s=1}^{d-1}\sigma_{s/d}(K)\vert \leq 2dg.
			\end{equation*}
	\end{enumerate}
\end{thm}

Here, $\sigma_{s/d}(K)$ denotes the Levine-Tristram signature of $K$ associated to $e^{2\pi is/d} \in S^1$. The signature $\sigma_1(\tau(K, \chi))$ is sometimes referred to as the Casson-Gordon $\tau$-signature of $K$.\footnote{Not to be confused with the Casson-Gordon signature, which is usually reserved for the invariants defined in \cite{cg2}.}

It will be convenient for our purposes to consider only elements in\break $H^1(X_d; \mathbb{Q}/\mathbb{Z})$ (resp.\ $H$) that have prime order. For this, let $p$ be any prime (for example one that divides the order of $H^1(X_d; \mathbb{Q}/\mathbb{Z})$), and consider the vector space $H^1(X_d; \mathbb{F}_p)$. Using the canonical embedding $\psi\colon \mathbb{F}_p \rightarrow \mathbb{Q}/\mathbb{Z}$, defined by $1 \mapsto \frac{1}{p}$, we obtain a diagram that defines an embedding $\varphi$ of $H^1(X_d; \mathbb{F}_p)$ into $H^1(X_d; \mathbb{Q}/\mathbb{Z})$:
\begin{center}
\begin{tikzcd}
	H^1(X_d; \mathbb{Q}/\mathbb{Z}) \arrow[r, "\cong"] & \textrm{Hom}(H_1(X_d; \mathbb{Z}), \mathbb{Q}/\mathbb{Z}) \\
	H^1(X_d; \mathbb{F}_p) \arrow[u, hook', "\varphi"] \arrow[r, "\cong"] & \textrm{Hom}(H_1(X_d; \mathbb{Z}), \mathbb{F}_p) \arrow[u, hook', "\widetilde\psi"'].
\end{tikzcd}
\end{center}
Here, the two horizontal isomorphisms are given by the universal coefficient theorem and $\widetilde\psi$ is the map induced by postcomposition with $\psi$. As described in \cref{seclfm}, the linking form $\beta$ now induces a linking form $\overline\beta$ on $H^1(X_d; \mathbb{F}_p)$ by setting $\overline\beta \coloneqq \beta \circ (\varphi \times \varphi)$. Note that this form will in general no longer be non-degenerate, and that $\varphi(H^1(X_d; \mathbb{F}_p))$ forms the entire $p$-torsion subgroup of $H^1(X_d; \mathbb{Q}/\mathbb{Z})$. Since $\varphi$ is injective, we can uniquely identify the elements of $H^1(X_d; \mathbb{F}_p)$ with the elements of $H^1(X_d; \mathbb{Q}/\mathbb{Z})$ of order $p$. In particular, we obtain a well-defined Casson-Gordon $\tau$-signature for the elements in $H^1(X_d;\mathbb{F}_p)$ by setting $\sigma_1(\tau(K, \chi)) \coloneqq \sigma_1(\tau(K, \varphi(\chi))$ for $\chi \in H^1(X_d; \mathbb{F}_p)$. This allows us to translate \cref{Gilmer} into this setting as follows.

\begin{cor}\label{Gilmer2}
	Let $K$ be a knot with $g_4(K) = g$ and $p$ a prime. Then $(H^1(X_d; \mathbb{F}_p), \overline\beta)$ splits as a direct sum $(G_1 \oplus G_2, \gamma_1 \oplus \gamma_2)$ such that
	\begin{enumerate}
		\item the dimension of $G_1$ over $\mathbb{F}_p$ is at most $2(d-1)g$;
		\item $\gamma_2$ has a generalized metabolizer $F$ such that for every $\chi \in F$,
			\begin{equation*}
				\vert \sigma_1(\tau(K, \chi)) + \sum_{s=1}^{d-1}\sigma_{s/d}(K)\vert \leq 2dg.\tag{$*$}\label{eq:gilmercor}
			\end{equation*}
	\end{enumerate}
\end{cor}

The most important difference is that the inequality \cref{eq:gilmercor} now holds for every $\chi \in F$, meaning that we no longer have to make the distinction between elements of arbitrary order and prime-power order.

We conclude this section with a short remark about connected sums. As before, let $X_d$ be the $d$-fold branched cover of $K$ where $d$ is a prime-power and consider the $n$-fold connected sum $nK$. The $d$-fold branched cover of $nK$ is a connected sum of $n$ copies of $X_d$, and the first cohomology splits accordingly. Litherland \cite{litherland2} showed that if $\chi = (\chi_1, \dots, \chi_n) \in H^1(nX_d; \mathbb{Q}/\mathbb{Z})$, then
\begin{equation*}
	\tau(nK, \chi) = \bigoplus_{i=1}^n \tau(K, \chi_i).
\end{equation*}
In particular,
\begin{equation*}
	\sigma_\lambda(\tau(nK, \chi)) = \sum_{i=1}^n \sigma_\lambda(\tau(K, \chi_i))
\end{equation*}
since $\sigma_\lambda$ is a homomorphism. In short, Casson-Gordon invariants and $\tau$-signatures behave well under connected sum. Note that the above also holds with $\mathbb{F}_p$-coefficients instead of $\mathbb{Q}/\mathbb{Z}$-coefficients.

\section{Main Result}\label{secmain}

Let $K$ be a knot with $d$-fold branched cover $X_d$, where $d$ is a prime-power, and $p$ a prime. Consider the vector space $H^1(X_d; \mathbb{F}_p)$. Every element $\chi \in H^1(X_d; \mathbb{F}_p)$ can be considered as a map $\chi\colon H_1(X_d; \mathbb{Z}) \rightarrow \mathbb{F}_p$ via universal coefficients, and to each such element there is the Casson-Gordon $\tau$-signature $\sigma_1(\tau(K, \chi))$ as mentioned in \cref{seccginv}. Since $H^1(X_d; \mathbb{F}_p)$ is finite, there are only finitely many Casson-Gordon $\tau$-signatures. We make the following definition.

\begin{dfn}\label{maindef}
	Assume that $H^1(X_d; \mathbb{F}_p)$ is non-trivial. Let $A_1, A_2, \dots, A_m$ be the one-dimensional subspaces of $H^1(X_d; \mathbb{F}_p)$. Define, for $j = 1, \dots, m$,
	\begin{equation*}
		L_j \coloneqq \sum_{\chi \in A_j} \sigma_1(\tau(K, \chi)) \in \mathbb{Q}.
	\end{equation*}
	Moreover, we set $L \coloneqq \min_{j=1,\dots,m} \vert L_j\vert$. If $H^1(X_d; \mathbb{F}_p)$ is trivial, we define $L \coloneqq 0$.
\end{dfn}
Note that $m = (p^t-1)/(p-1)$ with $t \coloneqq \dim H^1(X_d; \mathbb{F}_p)$ whenever $H^1(X_d; \mathbb{F}_p)$ is non-trivial.

Before we move on, let us quickly prove an introductory result as a warm-up exercise. This will not only show one of the key ideas used later on in a simplified context, but also allow the reader who would like to skip the (quite lengthy) proof of the main result to not miss out on the main proof technique.

\begin{prop}\label{propconcorder}
	Let $K$ be a knot and $p$ a prime such that $H^1(X_d; \mathbb{F}_p)$ is one-dimensional. If $L = \vert L_1\vert > 0$, then $g_4(nK) \neq 0$ for all $n \in \mathbb{N}$.
\end{prop}

\begin{proof}
	Let $n \in \mathbb{N}$ be a natural number and consider the $n$-fold connected sum $nK$. We would like to show with the given assumptions that $nK$ is not slice, i.e.\ $g_4(nK) \neq 0$.
	
	If any of the Levine-Tristram signatures $\sigma_{s/d}(K)$ is non-zero, then the Murasugi-Tristram bound $g_4(nK) \geq \frac{n}{2}\vert\sigma_{s/d}(K)\vert$ holds \cite{murasugi3, tristram}, so we may assume in the following that $\sum_{s=1}^{d-1}\sigma_{s/d}(K) = 0$. Let $\chi \in H^1(nX_d; \mathbb{F}_p)$ be a non-zero element. We claim that there exists a $k \in \mathbb{Z}$ such that \break$\sigma_1(\tau(nK, k \cdot \chi)) \neq 0$. Indeed, take the canonical basis on $H^1(nX_d; \mathbb{F}_p)$ given by the decomposition
	\begin{equation*}
		H^1(nX_d; \mathbb{F}_p) = \underbrace{H^1(X_d; \mathbb{F}_p) \oplus \dots \oplus H^1(X_d; \mathbb{F}_p)}_{n \textrm{ times}}
	\end{equation*}
	and let $R$ be the number of non-zero components of $\chi$ with respect to this basis. Then
	\begin{equation*}
		\bigg\vert\sum_{\ell = 1}^{p-1} \sigma_1(\tau(nK, \ell \cdot \chi))\bigg\vert = R\cdot L > 0.
	\end{equation*}
	Now, by \cref{Gilmer2}, if $nK$ was slice, there would exist a subspace $F \subset H^1(nX_d; \mathbb{F}_p)$ consisting only of elements with vanishing Casson-Gordon $\tau$-signature. However, we have just shown that every non-zero element in $H^1(nX_d; \mathbb{F}_p)$ has a multiple with non-zero $\tau$-signature. Thus $nK$ is not slice, and the result follows.
\end{proof}

Let us now get back to the proof of our main result. We start with the following technical proposition.

\begin{prop}\label{propadd}
	Let $K$ be a knot. If the rational numbers $L_1, L_2, \dots, L_m$ have the same sign and $L \neq 0$, then for any given $g \in \mathbb{N}$, there exists some $N \in \mathbb{N}$ such that $g_4(nK) > g$ for all $n \geq N$.
\end{prop}

\begin{proof}
	Fix some $g \in \mathbb{N}$ and consider the connected sum $nK$ for some $n \in \mathbb{N}$. If one of the Levine-Tristram signatures $\sigma_{s/d}(K)$ is non-zero, then the Murasugi-Tristram bound $g_4(nK) \geq \frac{n}{2}\vert\sigma_{s/d}(K)\vert$ holds \cite{murasugi3, tristram}, so we may assume in the following that $\sum_{s=1}^{d-1}\sigma_{s/d}(K) = 0$.

	Recall \cref{Gilmer2}: if $g_4(nK) = g$, then $(H^1(nX_d; \mathbb{F}_p), \overline\beta)$ splits as a direct sum $(G_1 \oplus G_2, \gamma_1 \oplus \gamma_2$), where the dimension of $G_1$ over $\mathbb{F}_p$ is at most $2(d-1)g$, and $G_2$ has a generalized metabolizer $F$ such that for all\break$\chi = (\chi_1, \dots, \chi_n) \in F$,
	\begin{equation*}
		\vert\sigma_1(\tau(nK, \chi))\vert = \bigg\vert \sum_{i=1}^n \sigma_1(\tau(K, \chi_i))\bigg\vert \leq 2dg.
	\end{equation*}
	Our goal is to show that by choosing $n$ appropriately, we can find an element $\chi \in F$ such that this inequality does not hold.
	
	Observe the following. Since $\dim G_1 \leq 2(d-1)g$ and $g$ is fixed, the dimension of $G_1$ is bounded when increasing $n$. On the other hand, increasing $n$ increases the dimension of $G_2$ since $\dim G_1 + \dim G_2 = \dim H^1(nX_d; \mathbb{F}_p)$, and with it the dimension of $F$.
	
	The Gauss-Jordan algorithm shows that if $\dim F = r$, then there is at least one element $\widetilde\chi = (\widetilde\chi_1, \dots, \widetilde\chi_n) \in F$ with $R \geq r$ non-zero entries. We may assume without loss of generality that the first $R$ entries are non-zero, i.e.\
	\begin{equation*}
		\widetilde\chi = (\widetilde\chi_1, \dots, \widetilde\chi_R, \underbrace{0, \dots, 0}_{n-R}).
	\end{equation*}
	From now on, we fix this element $\widetilde\chi$. Note that $k \cdot \widetilde\chi \in F$ for any $k \in \mathbb{N}$ since $F$ is a subspace. We are now going to show that there exists a multiple $k \cdot \widetilde\chi$ whose Casson-Gordon $\tau$-signature can be bounded from below by a value that depends on $n$.
	
	As in \cref{maindef}, let $A_1, \dots, A_m$ be the one-dimensional subspaces of $H^1(X_d; \mathbb{F}_p)$. Write
	\begin{equation*}
		A_j = \{0, \phi_1^j, \dots, \phi_{p-1}^j\} \subseteq H^1(X_d; \mathbb{F}_p)
	\end{equation*}
	for $j = 1, \dots, m$. For every $\phi_i^j$ there is an associated Casson-Gordon $\tau$-signature $s_i^j \coloneqq \sigma_1(\tau(K, \phi_i^j))$. In particular, for every component $\widetilde\chi_h$ for $h = 1, \dots, R$, there is some $j \in \{1, \dots, m\}$ and $i \in \{1, \dots, p-1\}$ such that
	\begin{equation*}
		\widetilde\chi_h = \phi_i^j \in A_j,\quad \sigma_1(\tau(K, \widetilde\chi_h)) = s_i^j.
	\end{equation*}
	Let $a_i^j$ denote the number of components in $\widetilde\chi$ with $\tau$-signature $s_i^j$ and set $r_j \coloneqq \sum_{i=1}^{p-1} a_i^j$. Observe that
	\begin{equation*}
		a_i^j \geq 0,\qquad \sum_{i=1}^{p-1} s_i^j = L_j,\qquad \sum_{j=1}^m r_j = R
	\end{equation*}
	for all $i = 1, \dots, p-1$ and $j = 1, \dots, m$. Recall from \cref{maindef} that we defined $L \coloneqq \min_{j=1,\dots,m} \vert L_j\vert$. We claim the following.
	\newline
	\newline
	\textbf{Claim.} There exists some $1 \leq k \leq p-1$ such that
	\begin{equation*}
		\bigg\vert\sum_{i=1}^n \sigma_1(\tau(K, k \cdot \widetilde\chi_i))\bigg\vert \geq \frac{R}{p-1} \cdot L.
	\end{equation*}
	
	\begin{proof} Consider the elements
	\begin{equation*}
	\begin{split}
		\widetilde\chi & = \phantom{\cdot}(\widetilde\chi_1, \dots, \widetilde\chi_R, 0 \dots, 0) \\
		2\widetilde\chi & = 2\cdot(\widetilde\chi_1, \dots, \widetilde\chi_R, 0, \dots, 0) \\
		& \ \vdots \\
		(p-1)\widetilde\chi & = (p-1)\cdot(\widetilde\chi_1, \dots, \widetilde\chi_R, 0, \dots, 0).
	\end{split}
	\end{equation*}
	For $\ell = 1, \dots, p-1$, let $a_{i, \ell}^j$ denote the number of components in $\ell \cdot \widetilde\chi$ with $\tau$-signature $s_i^j$. Since every $A_j$ is one-dimensional, the numbers $a_{1,\ell}^j, \dots, a_{p-1,\ell}^j$ are just a permutation of $a_1^j, \dots, a_{p-1}^j$ for every $j$. In fact, by looking at the multiples $\ell \cdot \widetilde\chi$ for $\ell = 1, \dots, p-1$, we have that for any given $a_i^j$ and $s_h^j$, there is an $\ell$ such that $a_i^j$ is the number of components in $\ell \cdot \widetilde\chi$ corresponding to the $\tau$-signature $s_h^j$. Therefore,
	\begin{equation*}
		\sum_{\ell=1}^{p-1}\sum_{j=1}^m\sum_{i=1}^{p-1}a_{i,\ell}^js_i^j = \sum_{j=1}^m\sum_{i=1}^{p-1}\bigg(\sum_{\ell = 1}^{p-1}a_{i,\ell}^j\bigg)s_i^j = \sum_{j=1}^m\sum_{i=1}^{p-1}r_js_i^j = \sum_{j=1}^m r_jL_j
	\end{equation*}
	Since also
	\begin{equation*}
		(p-1)\sum_{j=1}^m\sum_{i=1}^{p-1} \frac{r_j}{p-1} s_i^j = \sum_{j=1}^m r_jL_j,
	\end{equation*}
	we have that
	\begin{equation*}
		\sum_{\ell=1}^{p-1}\bigg(\sum_{j=1}^m\sum_{i=1}^{p-1} \Big(a_{i,\ell}^j s_i^j - \frac{r_j}{p-1} s_i\Big)\bigg) = 0.
	\end{equation*}
	Therefore, there have to be some $\ell_1, \ell_2 \in \{1, \dots, p-1\}$ such that
	\begin{equation*}
		\sum_{j=1}^m\sum_{i=1}^{p-1} \Big(a_{i,{\ell_1}}^j s_i^j - \frac{r_j}{p-1} s_i^j\Big) \geq 0 \quad \textrm{and} \quad \sum_{j=1}^m\sum_{i=1}^{p-1} \Big(a_{i,{\ell_2}}^j s_i^j - \frac{r_j}{p-1} s_i^j\Big) \leq 0.
	\end{equation*}
	Choose $k \in \{\ell_1, \ell_2\}$ such that
	\begin{equation*}
		\bigg\vert\sum_{j=1}^m\sum_{i=1}^{p-1} a_{i,k}^j s_i^j\bigg\vert \geq \bigg\vert\sum_{j=1}^m\sum_{i=1}^{p-1} \frac{r_j}{p-1} s_i^j\bigg\vert.
	\end{equation*}
	Unraveling the notations and using that all $L_j$ have the same sign, we find
	\begin{equation*}
	\begin{split}
		\bigg\vert\sum_{i=1}^{n} \sigma_1(\tau(K, k \cdot \widetilde\chi_i))\bigg\vert & = \bigg\vert\sum_{j=1}^m\sum_{i=1}^{p-1} a_{i,k}^j s_i^j\bigg\vert \\ &
		\geq \bigg\vert\sum_{j=1}^m\sum_{i=1}^{p-1} \frac{r_j}{p-1} s_i^j\bigg\vert \\ &
		= \bigg\vert\sum_{j=1}^m \frac{r_j}{p-1}L_j\bigg\vert \\ &
		\geq \bigg(\sum_{j=1}^m \frac{r_j}{p-1}\bigg) \cdot L \\ &
		= \frac{R}{p-1} \cdot L.\myqed
	\end{split}
	\end{equation*}
	\end{proof}
	\noindent The claim shows that there exists an element $k \cdot \widetilde\chi \in F$ such that
	\begin{equation*}
		\vert\sigma_1(\tau(nK, k \cdot \widetilde\chi))\vert \geq \frac{R}{p-1} \cdot L.
	\end{equation*}
	As mentioned earlier, the total number $R$ of non-zero components in $\widetilde\chi$ (or any multiple of it) depends on the dimension of $F$, which in turn depends on $n$, i.e.\ the number of summands in the connected sum $nK$. Since $g$ is fixed, adding more and more knots to the connected sum increases the dimension of $F$. Thus, choose $N \in \mathbb{N}$ such that the connected sum $NK$ admits a generalized metabolizer $F$ with dimension $r \leq R$ satisfying
	\begin{equation*}
		\frac{r}{p-1} \cdot L > 2dg.
	\end{equation*}
	Then for any $n \geq N$, the connected sum $nK$ admits an element $k \cdot \widetilde\chi \in F$ such that
	\begin{equation*}
		\vert\sigma_1(\tau(nK, k \cdot \widetilde\chi))\vert \geq \frac{R}{p-1} \cdot L > 2dg,
	\end{equation*}
	proving that $g_4(nK) > g$ by \cref{Gilmer2}.
\end{proof}

Before we continue, let us recall again \cref{maindef}: if $A_1, \dots, A_m$ are the one-dimensional subspaces of $H^1(X_d; \mathbb{F}_p)$, then we define
\begin{equation*}
	L_j \coloneqq \sum_{\chi \in A_j} \sigma_1(\tau(K, \chi)) \in \mathbb{Q},\quad j = 1, \dots, m.
\end{equation*}

\begin{prop}\label{mainprop}
	Let $K$ be a knot with $d$-fold branched cover $X_d$ where $d$ is a prime-power, and let $p$ be any prime. If the rational numbers $L_1, L_2, \dots, L_m$ have the same sign, and if $\sum_{s=1}^{d-1} \sigma_{s/d}(K) = 0$, where $\sigma_{s/d}(K)$ is the Levine-Tristram signature of $K$ associated to $e^{2\pi is/d}$, then
	\begin{equation*}
		g_4(nK) \geq \frac{nt\cdot L}{4d(p-1)+2(d-1)L}
	\end{equation*}
	for any $n \in \mathbb{N}$, where $t \coloneqq \dim H^1(X_d; \mathbb{F}_p)$ and $L \coloneqq \min_{j=1,\dots,m} \vert L_j\vert$.
\end{prop}

\begin{proof}
	We wish to determine the maximal number $g \in \mathbb{N}$ in terms of $n$ for which the proof of \cref{propadd} applies. So suppose that $H^1(nX_d; \mathbb{F}_p)$ splits as in \cref{Gilmer2} for some $g \in \mathbb{N}$, and let $r = \dim F$. If $L = 0$, then $nt\cdot L/(4d(p-1)+2(d-1)L) = 0$ and we obtain the trivial bound $g_4(nK) \geq 0$. So suppose that $L \neq 0$. The proof of \cref{propadd} showed that if
	\begin{equation*}
		\frac{r}{p-1}\cdot L > 2dg,
	\end{equation*}
	then $g_4(nK) \neq g$. Since $\dim G_2 \geq nt-2(d-1)g$, we know that $r \geq (nt-2(d-1)g)/2$, so
	\begin{equation*}
		\frac{r}{p-1}\cdot L \geq \frac{nt-2(d-1)g}{2(p-1)}\cdot L.
	\end{equation*}
	The right-hand side of the last inequality is strictly greater than $2dg$ if and only if
	\begin{equation*}
		g < \frac{nt\cdot L}{4d(p-1)+2(d-1)L}.
	\end{equation*}
	Thus, if
	\begin{equation*}
		g \leq \bigg\lceil\frac{nt\cdot L}{4d(p-1)+2(d-1)L}\bigg\rceil - 1,
	\end{equation*}
	then $g_4(nK) \neq g$ by applying the argument in the proof of \cref{propadd}. It follows that
	\begin{equation*}
		g_4(nK) \geq \frac{nt\cdot L}{4d(p-1)+2(d-1)L}
	\end{equation*}
	as claimed.
\end{proof}

\begin{thm}[Main Theorem]\label{mainthm}
	Let $K$ be a knot with $d$-fold branched cover $X_d$ where $d$ is a prime-power, and let $p$ be any prime. If the rational numbers $L_1, L_2, \dots, L_m$ have the same sign, and if $\sum_{s=1}^{d-1} \sigma_{s/d}(K) = 0$, where $\sigma_{s/d}(K)$ is the Levine-Tristram signature of $K$ associated to $e^{2\pi is/d}$, then
	\begin{equation*}
		g_{st}(K) \geq \frac{t\cdot L}{4d(p-1)+2(d-1)L},
	\end{equation*}
	where $t \coloneqq \dim H^1(X_d; \mathbb{F}_p)$ and $L \coloneqq \min_{j=1,\dots,m} \vert L_j\vert$.
\end{thm}

\begin{proof}
	By \cref{mainprop},
	\begin{equation*}
		g_4(nK) \geq \frac{nt\cdot L}{4d(p-1)+2(d-1)L}
	\end{equation*}
	for any $n \in \mathbb{N}$. Therefore
	\begin{equation*}
	\begin{split}
		g_{st}(K) & = \lim_{n\rightarrow\infty} \frac{g_4(nK)}{n} \\ &
		\geq \lim_{n\rightarrow\infty} \frac{1}{n}\cdot\frac{nt\cdot L}{4d(p-1)+2(d-1)L} \\ &
		= \frac{t\cdot L}{4d(p-1)+2(d-1)L}.\myqed
	\end{split} 
	\end{equation*}
\end{proof}

\section{Example: Twist Knots}\label{secex}

The results from \cref{secmain} are in particular applicable for (classical) genus one knots, for example the twist knots. This is because for genus one knots, there is an explicit formula, due to Gilmer, for computing the Casson-Gordon $\tau$-signature. We proceed by describing this formula in \cref{secexsig}, following the original source \cite{gilmer1}. In the subsequent \cref{secextwistsignature,secextwist}, we use this formula and the results from \cref{secmain} to obtain the lower bound for the stable $4$-genus of twist knots. In \cref{secexupperbound}, we apply a different technique recently used by Baader and Lewark \cite{bls4g} to obtain an upper bound for the stable $4$-genus of twist knots, yielding the subfamily with $g_{st}$ close to but not greater than $1/2$.

\subsection{Gilmer's formula for Casson-Gordon $\tau$-signatures}\label{secexsig}

The results in the following section are due to Gilmer \cite{gilmer1}. From now on, we will work exclusively with the double branched cover $X_2$ (see the remark at the end of \cref{secexsig}). Let $K$ be a knot with Seifert surface $F$,  Seifert pairing $\theta\colon H_1(F; \mathbb{Z}) \times H_1(F; \mathbb{Z}) \rightarrow \mathbb{Z}$, and double branched cover $X_2$. Define
\begin{equation*}
	\varepsilon\colon H_1(F; \mathbb{Z}) \rightarrow H^1(F; \mathbb{Z}),\quad x \mapsto \varepsilon_x(\cdot) = \theta(x, \cdot) + \theta(\cdot, x).
\end{equation*}
There is an isomorphism
\begin{equation*}
	H^1(X_2; \mathbb{Q}/\mathbb{Z}) \cong \ker (\varepsilon \otimes \textrm{id}_{\mathbb{Q}/\mathbb{Z}}),
\end{equation*}
which is natural up to sign. Given $\chi \in H^1(X_2; \mathbb{Q}/\mathbb{Z})$, this allows for the identification $\chi = x \otimes s/m$ for some $x \in H_1(F; \mathbb{Z})$ and $0 \leq s < m$. Given $\omega = e^{2\pi is/m} \in S^1$, let $\sigma_{s/m}(K)$ denote the Levine-Tristram signature of $K$ associated to $\omega$.

Suppose now that $g(F) = 1$, i.e.\ $K$ is of genus one. If $x \in H_1(F; \mathbb{Z})$ is primitive, let $J_x$ be the knot in $S^3$ obtained by representing $x$ by a simple closed curve $\gamma$ on $F$ and then viewing $\gamma$ as a knot in $S^3$. Note that $J_x$ is unique up to isotopy since $g(F) = 1$. We have the following theorem.

\begin{thm}[Gilmer \cite{gilmer1}]\label{thmex}
	Let $K$ be a genus one knot with genus-minimal Seifert surface $F$ and ordinary signature $\sigma(K)$. If $\chi = x\, \otimes\, s/m \in H_1(F; \mathbb{Q}/\mathbb{Z})$, where $0 < s < m$, $m$ is a prime-power and $x$ is primitive, then
	\begin{equation*}
		\sigma_1(\tau(K, \chi)) = 2\sigma_{s/m}(J_x) + \frac{4(m-s)s}{m^2}\theta(x,x) + \sigma(K).
	\end{equation*}
\end{thm}

Using \cref{thmex}, the computation of Casson-Gordon $\tau$-signatures for genus one knots boils down to the computation of generators of $\ker (\varepsilon \otimes \textrm{id}_{\mathbb{Q}/\mathbb{Z}})$, and then identifying the corresponding knots $J_x$ and their Levine-Tristram signature.\footnote{There is also a result about Casson-Gordon $\tau$-signatures for knots with higher genus; see \cite[Theorem 3.4]{gilmer1}. However, this result yields in general only an inequality for $\tau$-signatures.}

\begin{rem}
	Gilmer's formula for the Casson-Gordon $\tau$-signatures of genus one knots in \cite{gilmer1} is stated in terms of the double branched cover $X_2$. It is worth to note that generalizations of this formula to higher branched covers exist \cite{gilmer3, kimcg, naikcg}. However, since the double branched cover is the most accessible and since our formula for the stable $4$-genus yields a priori the best result for $d = 2$, we continue our computations for the twist knots with the double branched cover $X_2$.
\end{rem}

\subsection{Casson-Gordon $\tau$-signatures of twist knots}\label{secextwistsignature}

The main actors in the remaining sections are the twist knots. Given $n \in \mathbb{N}$, we will denote by $K_n$ the twist knot with $n$ full right hand twists, as depicted in \cref{fig:twist_knots_with_brace}.

\begin{figure}[H]
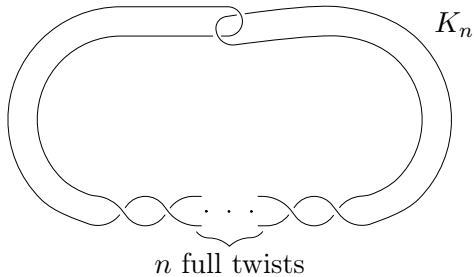

	\centering
	\captionsetup{font=small}
	\def\svgscale{0.24}
	
	\caption{The twist knot $K_n$}
	\label{fig:twist_knots_with_brace}
\end{figure}

The double branched cover of $K_n$ is the lens space $L(4n+1, 2)$, with $\mathbb{Q}/\mathbb{Z}$-(co-)homology
\begin{equation*}
	H_1(X_2; \mathbb{Q}/\mathbb{Z}) \cong H^1(X_2; \mathbb{Q}/\mathbb{Z}) \cong \mathbb{Z}_{4n+1}.
\end{equation*}
Given $\chi \in H^1(X_2; \mathbb{Q}/\mathbb{Z})$ of prime-power order, there is the Casson-Gordon invariant $\tau(K_n, \chi)$ and the Casson-Gordon $\tau$-signature $\sigma_1(\tau(K_n, \chi))$. Since $g(K_n) = 1$ for all $n \in \mathbb{N}$, we can use Gilmer's formula (see \cref{thmex}) to compute the $\tau$-signatures of the twist knots. We would like to note at this point that the following computations have already appeared in the literature previously in greater generality \cite{kimcg}. However, for the sake of completeness and because we use slightly different conventions, we chose to perform the computations once more.

Let $F_n$ be the genus-one Seifert surface for $K_n$ with $a$ and $b$ as a basis for the first homology as shown in \cref{fig:seifert_surface_twist_knots_homology_basis}.

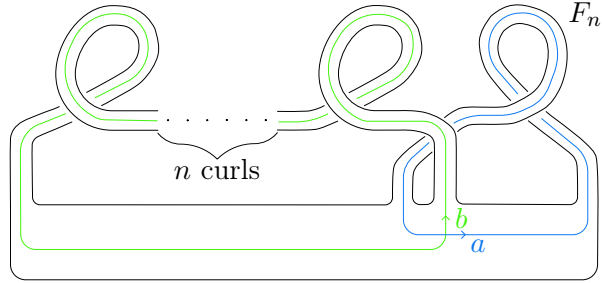
\begin{figure}[h]
	\centering
	\captionsetup{font=small}
	\def\svgscale{0.16}
	%% Creator: Inkscape 1.0beta2 (2b71d25, 2019-12-03), www.inkscape.org
%% Accompanies image file 'seifert_surface_twist_knots_homology_basis.pdf' (pdf, eps, ps)
\begingroup%
  \makeatletter%
  \providecommand\color[2][]{%
    \errmessage{(Inkscape) Color is used for the text in Inkscape, but the package 'color.sty' is not loaded}%
    \renewcommand\color[2][]{}%
  }%
  \providecommand\transparent[1]{%
    \errmessage{(Inkscape) Transparency is used (non-zero) for the text in Inkscape, but the package 'transparent.sty' is not loaded}%
    \renewcommand\transparent[1]{}%
  }%
  \providecommand\rotatebox[2]{#2}%
  \newcommand*\fsize{\dimexpr\f@size pt\relax}%
  \newcommand*\lineheight[1]{\fontsize{\fsize}{#1\fsize}\selectfont}%
  \ifx\svgwidth\undefined%
    \setlength{\unitlength}{1408.90777588bp}%
    \ifx\svgscale\undefined%
      \relax%
    \else%
      \setlength{\unitlength}{\unitlength * \real{\svgscale}}%
    \fi%
  \else%
    \setlength{\unitlength}{\svgwidth}%
  \fi%
  \global\let\svgwidth\undefined%
  \global\let\svgscale\undefined%
  \makeatother%
  \begin{picture}(1,0.46106837)%
    \lineheight{1}%
    \setlength\tabcolsep{0pt}%
    \put(0,0){\includegraphics[width=\unitlength,page=1]{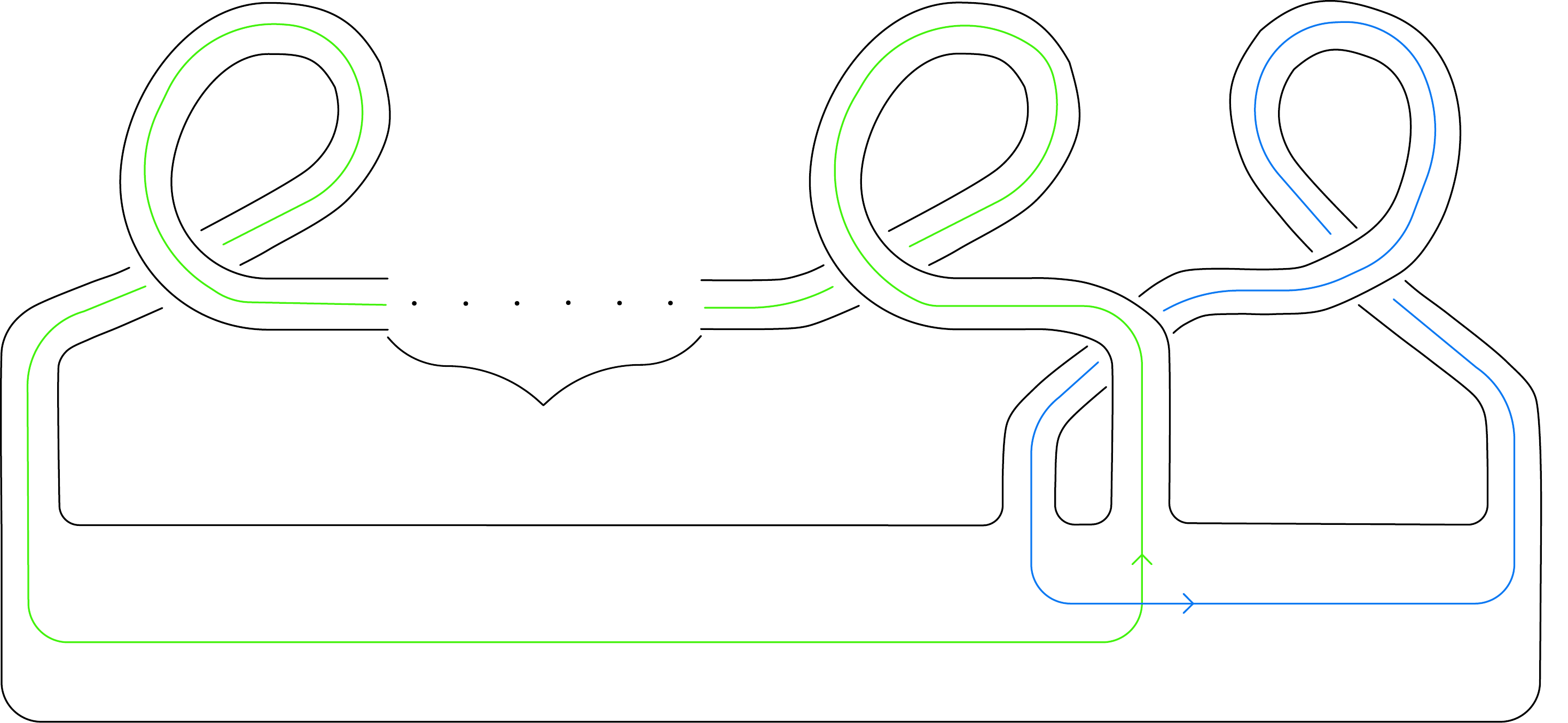}}%
    \put(0.27414834,0.16882058){\color[rgb]{0,0,0}\makebox(0,0)[lt]{\lineheight{1.25}\smash{\begin{tabular}[t]{l}$n$ curls\end{tabular}}}}%
    \put(0.9337018,0.43232257){\color[rgb]{0,0,0}\makebox(0,0)[lt]{\lineheight{1.25}\smash{\begin{tabular}[t]{l}$F_n$\end{tabular}}}}%
    \put(0.76974501,0.04691764){\color[rgb]{0,0,0}\makebox(0,0)[lt]{\lineheight{1.25}\smash{\begin{tabular}[t]{l}$\textcolor[HTML]{0D78F2}{a}$\end{tabular}}}}%
    \put(0.7441933,0.08737451){\color[rgb]{0,0,0}\makebox(0,0)[lt]{\lineheight{1.25}\smash{\begin{tabular}[t]{l}$\textcolor[HTML]{42F20D}{b}$\end{tabular}}}}%
  \end{picture}%
\endgroup%

	\caption{The Seifert surface $F_n$ for $K_n$ with a basis for $H_1(F_n; \mathbb{Z})$}
	\label{fig:seifert_surface_twist_knots_homology_basis}
\end{figure}

In this setting, the Seifert matrix of $K_n$ takes the form
\begin{equation*}
	A_n = \left(\begin{matrix}
		-1 & 1 \\
		0 & n
	\end{matrix}\right)
\end{equation*}
We are interested in finding generators for $\ker (\varepsilon \otimes \textrm{id}_{\mathbb{Q}/\mathbb{Z}}) \cong H^1(X_2; \mathbb{Q}/\mathbb{Z})$. Since $H^1(X_2; \mathbb{Q}/\mathbb{Z}) \cong \mathbb{Z}_{4n+1}$ is a finite cyclic group of order $4n+1$, any element of order $4n+1$ in this kernel will form a generating set. Let $x = (1, 2)^\top \in H_1(F_n; \mathbb{Z})$. We claim that the element
\begin{equation*}
	x \otimes \frac{1}{4n+1} = \left(\begin{matrix}
 1 \\ 2	
 \end{matrix}\right) \otimes \frac{1}{4n+1} \in H_1(F; \mathbb{Z}) \otimes \mathbb{Q}/\mathbb{Z}
\end{equation*}
is of order $4n+1$ and contained in $\ker (\varepsilon \otimes \textrm{id}_{\mathbb{Q}/\mathbb{Z}})$. The former assertion is clear. To see the latter, note that given any $y = (y_1, y_2)^\top \in H_1(F_n; \mathbb{Z})$,
\begin{equation*}
	\varepsilon_x(y) = x^\top A_n y + y^\top A_n x = (4n+1)y_2,
\end{equation*}
showing that $(\varepsilon \otimes \textrm{id}_{\mathbb{Q}/\mathbb{Z}})(x \otimes 1/(4n+1))$ is zero in $H^1(F_n; \mathbb{Z}) \otimes \mathbb{Q}/\mathbb{Z}$.

The next step consists of representing $x = (1, 2)^\top$ as a simple closed curve on $F_n$ and determining what knot $J_x$ this curve represents in $S^3$. \cref{fig:torus_knot} shows this process. It turns out that the element $x$ represents a $(2, 2n+1)$-torus knot in $S^3$, that is $J_x = T(2, 2n+1)$.

\begin{figure}[H]
	\centering
	\captionsetup{font=small}
	\fontsize{8}{10}\selectfont
	\def\svgscale{0.143}
	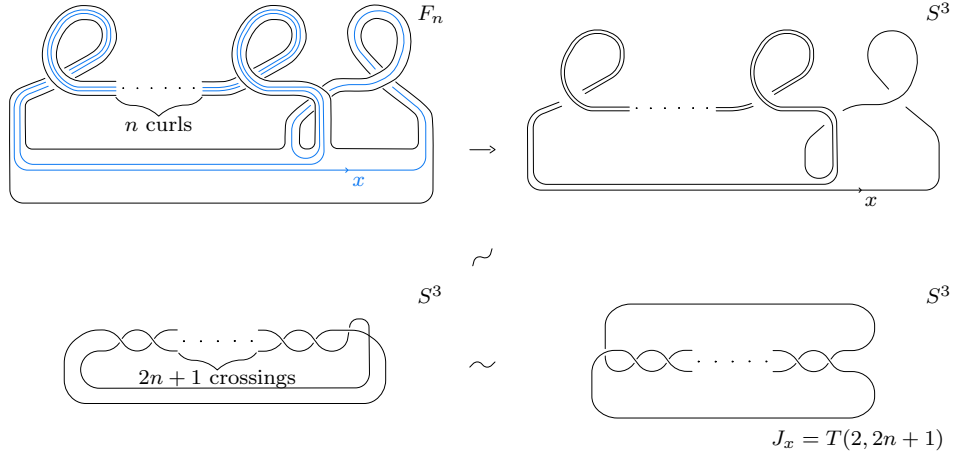
	\caption{The element $x = (1, 2)^\top \in H_1(F_n; \mathbb{Z})$ represents a $(2, 2n+1)$-torus knot in $S^3$}
	\label{fig:torus_knot}
\end{figure}

We now have all of the ingredients to compute the Casson-Gordon $\tau$-signatures from \cref{thmex}. Assume for the moment that $4n+1$ is a power of a prime, so that every element in $H^1(X_2; \mathbb{Q}/\mathbb{Z})$ has prime-power order. Given
\begin{equation*}
	\chi = x \otimes \frac{s}{4n+1} \in H^1(X_2; \mathbb{Q}/\mathbb{Z}),\quad 0 < s < 4n+1,
\end{equation*}
we get
\begin{align*}
	\sigma_1(\tau(K_n, \chi)) & = \sigma_{s/(4n+1)}(T(2, 2n+1)) + \frac{4((4n+1)-s)s}{(4n+1)^2}\theta(x,x) + \sigma(K_n) \\ &
	= \sigma_{s/(4n+1)}(T(2, 2n+1)) + \frac{4((4n+1)-s)s}{4n+1}
\end{align*}
since $\theta(x,x) = 4n+1$ and $\sigma(K_n) = 0$ for all $n \in \mathbb{N}$. The Levine-Tristram signatures of $(2, 2n+1)$-torus knots are well-known and readily computed (see for instance \cite{litherland}). Thus overall,
\begin{equation}\label{eq:tkcgtausign}
	\sigma_1(\tau(K_n, \chi)) = \begin{cases}
		-4\left\lceil \frac{s}{2} \right\rceil + \frac{4((4n+1)-s)s}{4n+1},& s = 1, \dots, 2n \\
		-4\left\lceil \frac{4n+1-s}{2} \right\rceil + \frac{4((4n+1)-s)s}{4n+1},& s = 2n+1, \dots, 4n.
	\end{cases}
\end{equation}
\cref{tab:table3} shows the Casson-Gordon $\tau$-signatures for the twist knot $K_6$. Notice the symmetry of the values about $12$.

\begin{table}[H]
	\centering
	\captionsetup{width=.8\linewidth, font=small}
	%\resizebox{0.6\textwidth}{!}{%
	\begin{tabular}{cc}
	Element & Signature \\ 
	\hline
	$0$ & $0$ \\
	$1$ & $-0.16$ \\
	$2$ & $3.36$ \\
	$3$ & $2.56$ \\
	$4$ & $5.44$ \\
	$5$ & $4$ \\
	$6$ & $6.24$ \\
	$7$ & $4.16$ \\
	$8$ & $5.76$ \\
	$9$ & $3.04$ \\
	$10$ & $4$ \\
	$11$ & $0.64$ \\
	$12$ & $0.96$
	\end{tabular}%
	%}
	\quad
	\begin{tabular}{cc}
	Element & Signature \\ 
	\hline
	$13$ & $0.96$ \\
	$14$ & $0.64$ \\
	$15$ & $4$ \\
	$16$ & $3.04$ \\
	$17$ & $5.76$ \\
	$18$ & $4.16$ \\
	$19$ & $6.24$ \\
	$20$ & $4$ \\
	$21$ & $5.44$ \\
	$22$ & $2.56$ \\
	$23$ & $3.36$ \\
	$24$ & $-0.16$ \\
	\null
	\end{tabular}
	\caption{Elements in $H^1(X_2; \mathbb{Q}/\mathbb{Z}) \cong \mathbb{Z}_{25}$ and their corresponding Casson-Gordon $\tau$-signatures for the twist knot $K_6$.}
	\label{tab:table3}
\end{table}

Suppose now that $4n+1$ is a general number, not necessarily a prime-power. It is still true that the right-hand side of the equation for the Casson-Gordon $\tau$-signature given by \cref{thmex} is equal to the right-hand side of \cref{eq:tkcgtausign} for twist knots. However, the whole equation \cref{eq:tkcgtausign} only holds if the corresponding character $\chi$ has prime-power order. So the best way to get information about the Casson-Gordon $\tau$-signature of an arbitrary twist knot is to study the right-hand side of \cref{eq:tkcgtausign}, and keep in mind that it is equal to $\sigma_1(\tau(K_n, \chi))$ if and only if the corresponding character has prime-power order.

In general, the values of the two formulas for $\sigma_1(\tau(K_n, \chi))$ in \cref{eq:tkcgtausign} are symmetric about $2n$. Thus, it suffices to consider only the first formula and take symmetry into account. We will do so in the following section.

\subsection{Lower bound for the stable $4$-genus of twist knots}\label{secextwist}

We proceed by computing the lower bound for $g_{st}(K_n)$ from our results in \cref{secmain}. Let $p$ be a prime dividing $4n+1$. Since $H^1(X_2; \mathbb{Q}/\mathbb{Z})$ is finite cyclic, $H^1(X_2; \mathbb{F}_p)$ is one-dimensional; hence there is only one number
\begin{equation*}
	L_1 = \sum_{\chi \in H^1(X_2; \mathbb{F}_p)} \sigma_1(\tau(K, \chi)).
\end{equation*}
In particular, $\vert L_1\vert = L$. Observe that if $p$ appears with exponent $k$ in the prime-decomposition of $4n+1$, then the elements of $H^1(X_2; \mathbb{F}_p)$, considered as elements in $\mathbb{Z}_{4n+1}$, are $0, q, 2q, \dots, (p-1)q$ with $q = (4n+1)/p$. Then, using the symmetry of the Casson-Gordon $\tau$-signatures for the twist knots,
\begin{equation}\label{eq:tkl1}
\begin{aligned}[b]
	L_1 & = 2\sum_{s=1}^{(p-1)/2} -4\left\lceil\frac{sq}{2}\right\rceil + \frac{4((4n+1)-sq)sq}{4n+1} \\ &
	= q(p-1)(p+1) - \bigg(\frac{pq^2(p-1)(p+1)}{3(4n+1)}\bigg) - \sum_{s=1}^{(p-1)/2} 8\left\lceil\frac{sq}{2}\right\rceil \\ &
	= \frac{2}{3}q(p^2-1) - \sum_{s=1}^{(p-1)/2} 8\left\lceil\frac{sq}{2}\right\rceil.
\end{aligned}
\end{equation}
To evaluate the remaining sum, we distinguish two cases. If $(p-1)/2$ is even, then
\begin{equation*}
\begin{split}
	\sum_{s=1}^{(p-1)/2} 8\left\lceil\frac{sq}{2}\right\rceil & = \sum_{i=1}^{(p-1)/4} 8\bigg(iq + \frac{(2i-1)q+1}{2}\bigg) \\ &
	= (1-q)(p-1) + \frac{q(p-1)(p+3)}{2} \\ &
	= \frac{1}{2}(p-1)(pq+q+2).
\end{split}
\end{equation*}
On the other hand, if $(p-1)/2$ is odd, then
\begin{equation*}
\begin{split}
	\sum_{s=1}^{(p-1)/2} 8\left\lceil\frac{sq}{2}\right\rceil & = \bigg(\sum_{i=1}^{(p-3)/4} 8\Big(iq + \frac{(2i-1)q+1}{2}\Big)\bigg) + 4q\bigg(\frac{p-1}{2}+1\bigg) \\ &
	= \frac{q(p-3)(p+1)}{2}+(p-3)(1-q)+2q(p-1)+8 \\ &
	= \frac{1}{2}(p^2q+2p-q+2).
\end{split}
\end{equation*}
Putting this into the equation \cref{eq:tkl1} above and simplifying further, we overall obtain
\begin{equation*}
	L_1 = \begin{cases}
		\frac{1}{6}(p-1)(pq+q-6), & (p-1)/2 \textrm{ even } \\
		\frac{1}{6}(p^2q-6p-q-6), & (p-1)/2 \textrm{ odd }.
	\end{cases}
\end{equation*}
Note that $p = 2$ cannot occur since $4n+1$ is always odd, so we covered all cases above. It is not hard to see that $L_1 \geq 0$ for all primes $p \geq 3$. Moreover, all Levine-Tristram signatures of twist knots vanish, hence we can apply \cref{mainthm} and obtain the following.

\begin{cor}\label{corexmain}
	Let $K_n$ be the twist knot with $n \in \mathbb{N}\backslash\{0,2\}$ full right hand twists and $p$ a prime dividing $4n+1$. Then
	\begin{equation*}
		g_{st}(K_n) \geq \begin{cases}
		\frac{(pq+q-6)}{2(pq+q+18)}, & (p-1)/2 \textrm{ even } \\
		\frac{p^2q-6p-q-6}{2(p^2q+18p-q-30)}, & (p-1)/2 \textrm{ odd },
	\end{cases}
	\end{equation*}
	where $q = (4n+1)/p$.
\end{cor}

\begin{rem}
	One can check from the previous computations that $K_0$, $K_1$ and $K_2$ are the only twist knots with $L = 0$. This means that a non-trivial lower bound for $g_{st}$ can be obtained from \cref{corexmain} (resp.\ \cref{mainthm}) for any twist knot $K_n$ with $n \geq 3$. The unknot $K_0$ and the Stevedore knot $K_2$ are slice, and the figure-eight $K_1$ represents torsion in the knot concordance group $\mathcal{C}$. Thus we obtain the following corollaries.
\end{rem}
	
\begin{cor}\label{corextorsion}
	Let $K_n$ be any twist knot. Then
	\begin{equation*}
		K_n \textrm{ is torsion in $\mathcal{C}$} \iff g_{st}(K_n) = 0.
	\end{equation*}
\end{cor}

\begin{cor}\label{corexconcorder}
	$K_n$ is of infinite order in $\mathcal{C}$ except for $n = 0, 1, 2$.
\end{cor}

For arbitrary knots $K$, it is an open question whether $g_{st}(K) = 0$ implies that $K$ is torsion in the knot concordance group $\mathcal{C}$.

It is important to note that the strength of the lower bound in \cref{corexmain} depends on the choice of prime factor $p$; a priori, there is no preferred choice of $p$ to obtain the strongest bound. In order to obtain the best result, one has to compute the lower bound for all primes $p$ in the prime decomposition of $4n+1$ and then compare. In \cref{tab:table4} we have computed the lower bounds given by \cref{corexmain} for the twist knots $K_5$, $K_{11}$, $K_{16}$, $K_{21}$ and $K_{400}$.

\begin{table}[H]
		\centering
		\captionsetup{width=.8\linewidth, font=small}
		%\resizebox{0.6\textwidth}{!}{%
		\begin{tabular}{c|c|c|c|c|c|c||c}
		$K_n$ $\backslash$ $p$ & $3$ & $5$ & $7$ & $13$ & $17$ & $89$ & $4n+1$ \\ 
		\hline
		$K_5$ & $1/5$ & $0$ & $1/5$ & $0$ & $0$ & $0$ & $21 = 3\cdot 7\phantom{0}$ \\
		\hline
		$K_{11}$ & $1/3$ & $1/3$ & $0$ & $0$ & $0$ & $0$ & $45 = 5\cdot 9\phantom{0}$ \\
		\hline
		$K_{16}$ & $0$ & $3/8$ & $0$ & $4/11$ & $0$ & $0$ & $65 = 5\cdot 13$ \\
		\hline
		$K_{21}$ & $0$ & $2/5$ & $0$ & $0$ & $7/18$ & $0$ & $85 = 5\cdot 17$ \\
		\hline
		$K_{400}$ & $22/45$ & $0$ & $0$ & $0$ & $0$ & $67/138$ & $801 = 9\cdot 89$ \\
		\hline
		\end{tabular}%
		%}
		\caption{Examples of lower bounds for the stable $4$-genus of twist knots obtained from \cref{mainthm}.}
		\label{tab:table4}
\end{table}

While the bounds in \cref{corexmain} are directly obtained from the main theorem and are the strongest that we currently know, they are not particularly easy to grasp. By estimating further we obtain a weaker result that holds for all twist knots simultaneously and is easier to grasp.

\begin{cor}[\cref{corexmain}, weakened]\label{corexadd}
	Let $K_n$ with $n \in \mathbb{N}$, be any twist knot. Then
	\begin{equation*}
		g_{st}(K_n) \geq \frac{1}{2} - \frac{6}{2n+7}.
	\end{equation*}
\end{cor}

\begin{prf}
	The result is obtained by estimating the two bounds given in \cref{corexmain}. Recall that $p$ is a prime dividing $4n+1$ and $q = (4n+1)/p$. Observe that $3 \leq p \leq 4n+1$, $1 \leq q \leq (4n+1)/3$, and $pq = 4n+1$.
	\begin{enumerate}[1.)]
		\item In the first case of \cref{corexmain},
			\begin{equation}\label{eq:corweak1}
			\begin{aligned}[b]
				g_{st}(K_n) & \geq \frac{pq+q-6}{2(pq+q+18)} \\ &
				= \frac{pq+q+18}{2(pq+q+18)} - \frac{24}{2(pq+q+18)} \\ &
				= \frac{1}{2} - \frac{12}{4n+q+19} \\ &
				\geq \frac{1}{2} - \frac{3}{n+5}.
			\end{aligned}
			\end{equation}
			The last inequality is obtained by estimating $q$ from below with $1$.
		\item In the second case of \cref{corexmain},
			\begin{equation}\label{eq:corweak2}
			\begin{aligned}[b]
				g_{st}(K_n) & \geq \frac{p^2q-6p-q-6}{2(p^2q+18p-q-30)} \\ &
				= \frac{1}{2} - \frac{12(p-1)}{p^2q+18p-q-30} \\ &
				= \frac{1}{2} - \frac{12(p-1)}{(4n+1)p+18p-q-30} \\ &
				= \frac{1}{2} - \frac{12}{4n+q+19-\frac{12}{p-1}} \\ &
				\geq \frac{1}{2} - \frac{6}{2n+7}.
			\end{aligned}
			\end{equation}
			The third equality is obtained by using
			\begin{equation*}
			\begin{split}
				(4n+1)p + 18p - q - 30 & = (4n+1)p + 18p - 4n+1 - 18 \\ &
				\quad\, + pq -q - 12 \\ &
				= (p-1)\Big(4n+1 + q + 18 - \frac{12}{p-1}\Big),
			\end{split}
			\end{equation*}
			while the last inequality is obtained by estimating $q$ from below with $1$ and $12/(p-1)$ from above with $12/2$ by using $p \geq 3$.
	\end{enumerate}
	Comparing the lower bounds obtained in \cref{eq:corweak1} and \cref{eq:corweak2}, we see that
	\begin{equation*}
		\frac{1}{2} - \frac{3}{n+5} \geq \frac{1}{2} - \frac{6}{2n+7}
	\end{equation*}
	for all $n \in \mathbb{N}$, and the result follows.\myqed
\end{prf}

It is immediate from \cref{corexadd} that for growing $n$, the bound tends towards $1/2$. This implies that the stronger bounds in \cref{corexmain} also tend towards $1/2$ as $n$ grows, since $1/2$ is an upper bound for the lower bound in the main theorem. We will use this fact in the next section to show that there exists an infinite subfamily of twist knots with stable $4$-genus close to but not greater than $1/2$.

\subsection{Twists knots with stable $4$-genus close to but not greater\break than $1/2$}\label{secexupperbound}

A recent result by Baader and Lewark \cite{bls4g} implies that\break $g_{st}(K_n) \leq 2/3$ for any $n \in \mathbb{N}$ (see \cite[Lemma 5]{bls4g}). The idea is to take the three-fold connected sum $3K_n$ with Seifert surface $3\Sigma$, where $\Sigma$ is a genus-minimal Seifert surface for $K_n$, and find a subgroup of rank two of  $H_1(3\Sigma; \mathbb{Z})$ on which the Seifert form has the matrix
\begin{equation}\label{eq:s4gub}
	\left(\begin{matrix}
		0 & 1 \\
		0 & c
	\end{matrix}\right)
\end{equation}
for some $c \in \mathbb{Z}$. As the proof of \cite[Lemma 5]{bls4g} shows, one can always find such a subgroup when taking at least three copies of $K_n$. In this setting, Baader and Lewark show that one can achieve a situation in which Freedman's disc theorem \cite{fdisk, fqbook} can be applied to obtain $g_4(3K_n) \leq g(3K_n) - 1 = 2$, implying that $g_{st}(K_n) \leq 2/3$.

A natural question that arises is under what conditions one can find such a subgroup of rank two starting with two copies of the knot instead of three. Thus, let $K_n$ be any twist knot with its standard genus-one Seifert surface $\Sigma$ and Seifert matrix
\begin{equation*}
	A = \left(\begin{matrix}
		1 & 1 \\
		0 & -n
	\end{matrix}\right).
\end{equation*}
The two-fold connected sum $2K_n$ has then $2\Sigma$ as a Seifert surface with Seifert matrix the block sum $\bar A \coloneqq A \oplus A$. Consider the vectors
\begin{equation*}
	v = (1, 0, x, y)^\top,\ w = (0, 1, 0, 0)^\top\in H_1(2\Sigma; \mathbb{Z}),
\end{equation*}
where $x, y \in \mathbb{Z}$ are yet to be found. Independent of the choice of $x$ and $y$, we have
\begin{equation*}
	v^\top \bar Aw = 1,\ w^\top \bar Av = 0,\ w^\top \bar Aw = -n.
\end{equation*}
Thus, it remains to find $x, y \in \mathbb{Z}$ such that $v^\top \bar Av = 0$; that is
\begin{equation*}
	v^\top \bar Av = x^2 + xy - ny^2 + 1 = 0.
\end{equation*}
Similar to the proof of \cite[Lemma 4]{bls4g}, we complete the square to obtain
\begin{equation*}
\begin{split}
	x^2 + xy - ny^2 & = \Big(x + \frac{y}{2}\Big)^2 - (4n+1)\Big(\frac{y}{2}\Big)^2 \\ &
	= \overline x^2 - (4n+1) \overline y^2,
\end{split}
\end{equation*}
where in the last equation we substituted $\overline y = y/2$ and $\overline x = x + \overline y$. We obtain
\begin{equation*}
	v^\top \bar Av = 0 \iff \overline x^2 - (4n+1) \overline y^2 = -1.
\end{equation*}
The equation $\overline x^2 - (4n+1) \overline y^2 = -1$ is generally known as a negative Pell equation. If this equation has a solution $\overline x, \overline y \in \mathbb{Z}$, then setting $x = \overline x - \overline y$ and $y = 2\overline y$ in $v$ gives us a vector such that $v^\top \bar Av = 0$. It follows that the Seifert form of $2K_n$ restricted to the rank-two subgroup spanned by $v$ and $w$ is of the form
\begin{equation*}
	\left(\begin{matrix}
		0 & 1 \\
		0 & c
	\end{matrix}\right)
\end{equation*}
for some $c \in \mathbb{Z}$. Proceeding as in the proof of \cite[Lemma 5]{bls4g}, we obtain
\begin{equation*}
	g_4(2K_n) \leq 1 \implies g_{st}(K_n) \leq \frac{1}{2},
\end{equation*}
provided that for $4n+1$, the negative Pell equation has an integer valued solution. We summarize the above observations in the following proposition.

\begin{prop}\label{propex}
	Let $n \in \mathbb{N}$ be such that the negative Pell equation\break $x^2 - (4n+1)y^2 = -1$ has a solution $x, y \in \mathbb{Z}$. Then
	\begin{equation*}
		g_{st}(K_n) \leq \frac{1}{2}.\myqed
	\end{equation*}
\end{prop}

As mentioned in the introduction, a necessary and sufficient condition for the existence of a solution of the negative Pell equation is that the continued fraction of $\sqrt{4n+1}$ has odd period length \cite{rtpell}. This is the case, for example, if $4n+1 = p^k$, where $p$ is a prime such that $p \equiv 1 \mod 4$ and $k \in \mathbb{N}$ \cite{rtpell}. This yields, together with the lower bound given in \cref{corexmain}, the infinite subfamily of twist knots with $g_{st}$ close to but not greater than $1/2$.

Note that in the approach above, we specified two explicit vectors $v$ and $w$ and derived a sufficient condition for them to span a rank-two subgroup on which the Seifert matrix has the desired form \cref{eq:s4gub}. A natural question to ask is  under what circumstances such a subgroup exists in general. Indeed, although the solvability of the negative Pell equation is a sufficient condition, it is not necessary. For example, if $n = 51$, then $4n+1 = 205$, but ${x^2 - 205y^2 = -1}$ has no solution. Yet, the Seifert form of $2K_{51}$ restricted to the rank-two subgroup spanned by the vectors $(13, 2, 3, 0)^\top$ and $(14, 2, -2, 1)^\top$ gives the desired matrix \cref{eq:s4gub}, so $g_{st}(K_{51}) \leq 1/2$.

More generally, there might also be other, different methods to obtain the upper bound given by $1/2$. We do not know of a full characterization of twist knots for which $g_{st} \leq 1/2$ holds.

\bibliographystyle{myamsalpha}
\bibliography{References.bib}

\begin{comment}

\end{comment}

\end{document}